\documentclass[12pt, leqno]{amsart}%fleqn
\usepackage{amsmath}
\usepackage{amssymb}
\usepackage{amsfonts}
\usepackage{graphicx}
\usepackage{amsthm}
\usepackage{enumerate}
\usepackage{lscape}
\usepackage{dsfont}
\usepackage{xcolor}
\usepackage{mathtools}
\usepackage{pgf,tikz,tikz-cd}
\usetikzlibrary{arrows}
\usepackage{color}

\newtheorem{theor}{Theorem}

\newtheorem{defin}[theor]{Definition}
\newtheorem{lemma}[theor]{Lemma}
\newtheorem{prop}[theor]{Proposition}

\newtheorem{ex}{Example}
\newtheorem{remark}[theor]{Remark}

\title{Projectively induced K\"ahler cones over regular Sasakian manifolds}
\author[Stefano Marini, Nicoletta Tardini and Michela Zedda]{Stefano Marini, Nicoletta Tardini and Michela Zedda}
\address{Dipartimento di Scienze Matematiche, Fisiche e Informatiche\\
Unit\`{a} di Matematica e Informatica,
Universit\`{a} degli Studi di Parma\\
Parco Area delle Scienze 53/A, 43124 \\
Parma, Italy}
\email{stefano.marini@unipr.it}
\email{nicoletta.tardini@unipr.it}
\email{michela.zedda@unipr.it}

\thanks{This research has been financially supported by the project Prin 2022 – Real and Complex Manifolds: Geometry and Holomorphic Dynamics – Italy, and by GNSAGA of INdAM}

\begin{document}
\maketitle
\begin{abstract}
Motivated by a conjecture in \cite{LSZ} we prove that the K\"ahler cone over a regular complete Sasakian manifold is Ricci--flat and projectively induced if and only if it is flat. We also obtain that, up to $\mathcal D_a$--homothetic transformations, K\"ahler cones over homogeneous compact Sasakian manifolds are projectively induced. As main tool we provide a relation  between the K\"ahler potentials of the transverse K\"ahler metric and of the cone metric. 
\end{abstract}

\section{Introduction}
In the celebrated work \cite{calabi}, E. Calabi gives an algebraic criterion to check when a K\"ahler manifold admits a K\"ahler (i.e. holomorphic and isometric) immersion into a complex space form. Of particular interest is the elliptic case, namely when the complex space form is the complex projective space endowed with the Fubini--Study metric, where many questions are still open. Throughout this paper $\mathds C{\rm P}^N$ is always implicitly understood, as a K\"ahler manifold, equipped with the Fubini-Study metric. We say that a K\"ahler manifold is {\em projectively induced} if it admits a {\em local} K\"ahler immersion into the complex projective space $\mathds C{\rm P}^{N\le\infty}$. 
Obviously many examples of projectively induced manifolds can be constructed pulling--back the Fubini--Study metric on complex submanifolds. However, it is much more difficult to find examples of projectively induced manifolds with prescribed curvature. In \cite{H} D. Hulin proves that the scalar curvature of a compact K\"ahler--Einstein submanifold of the complex projective space is forced to be positive.
 It is important to highlight that when the submanifold is compact the ambient space $\mathds C{\rm P}^N$ can be taken to be finite dimensional. In fact, when the ambient space is infinite-dimensional,  Hulin's result does not hold, as there are  examples of K\"ahler--Einstein submanifolds of $\mathds C{\rm P}^\infty$ with negative scalar curvature (see \cite{LZ}). Furthermore, the complex flat space $\mathds C^n$ is an example of K\"ahler submanifold of $\mathds C{\rm P}^\infty$. Observe that by the rigidity result of Calabi \cite[Theorem 9]{calabi}, $\mathds C^n$ does not admit a K\"ahler immersion in $\mathds C{\rm P}^{N<\infty}$. Actually, in the recent work \cite{ALL} C. Arezzo, C. Li and A. Loi prove that no $\mathds C{\rm P}^{N<\infty}$ admits Ricci--flat K\"ahler submanifolds. On the other hand, when the ambient space is infinite-dimensional, it is still an open problem to understand whether the flat metric is the only example of Ricci-flat projectively induced K\"ahler metric, as conjectured by A. Loi, F. Salis and F. Zuddas in \cite{LSZ} (see also \cite{LZZ} for a class of metrics that confirms such conjecture).

In this paper we address the problem of studying Sasakian manifolds whose K\"ahler cone is Ricci-flat and projectively induced.
We recall that the K\"ahler cone over a Sasakian manifold is Ricci--flat if and only if the Sasakian manifold is Sasaki--Einstein, and it is flat if and only if the Sasakian manifold is a standard sphere (see Theorem \ref{SE} below).  More precisely we prove the following.
\begin{theor}\label{main}
The K\"ahler cone over a regular complete Sasakian manifold is Ricci--flat and projectively induced if and only if it is flat.
\end{theor}

Theorem \ref{main} should be compared to the results in \cite{BCML, CML, LPZ}, where the existence of a Sasakian immersion of a Sasakian manifold into a Sasakian space form is investigated either in terms of the existence of a K\"ahler immersion of the K\"ahler cone above or in terms of the transverse K\"ahler metric below. A key step in their approach consists in observing that a compact Sasakian manifold immersed into a regular Sasakian manifold is itself regular, as follows from \cite[Prop. 3.1]{Harada} (see also \cite[Prop. 5]{BCML} for a generalization to the non--compact case). More precisely, they do not need to assume {\em a priori} the regularity of the Sasakian manifold since they obtain it by assuming a Sasakian immersion into a (regular) Sasakian space form. 
In particular, when the Sasakian space form is elliptic one can lift the Sasakian map to the K\"ahler cone in a natural way, obtaining that a Sasakian immersion into a sphere is possible if and only if a K\"ahler immersion of the K\"ahler cone into the complex Euclidean space occurs. When the Sasakian space form is not elliptic its K\"ahler cone is not a complex space form, but using regularity it is possible to relate the existence of a Sasakian immersion to a K\"ahler immersion of the transverse K\"ahler metric below. In our approach instead, we need to impose the regularity of the Sasakian manifold, since no Sasakian immersion is assumed to exist. More precisely, we cannot deduce the existence of a Sasakian immersion from the K\"ahler cone being projectively induced, since the complex projective space is not a cone over a Sasakian manifold.

The novelty of our paper consists in relating the existence of a K\"ahler immersion of the K\"ahler cone over a regular Sasakian manifold into the (infinite-dimensional) complex projective space, to the existence of a K\"ahler immersion of the transverse K\"ahler manifold below into a complex projective space (Proposition \ref{amelia} below).

Finally, we show that by performing a $\mathcal D_a$--homothetic transformation, namely a rescaling of the structure (see Formula \eqref{da} below), we can find many examples of projectively induced K\"ahler cones over homogeneous Sasakian compact manifolds.
More precisely, as a consequence of Proposition \ref{amelia} we get the following:
\begin{theor}\label{main2}
Let $(S,\xi,\eta,\Phi,g)$ be a homogeneous compact Sasakian manifold. Then, up to a $\mathcal D_a$--homothetic transformation, its K\"ahler cone is projectively induced.
\end{theor}
For other results concerning immersions of homogeneous Sasakian manifolds the reader is referred to \cite[Theorem 4]{CML}, where it is proved that up to a $\mathcal D_a$-homothetic deformation a compact homogeneous Sasakian manifold admits a Sasakian immersion into the Sasakian sphere if and only if its fundamental group is cyclic (see also \cite[Theorem 1.5]{LPZ} for a noncompact version). It is worth pointing out that from Theorem \ref{main2} we obtain examples of $\eta$--Einstein Sasakian manifolds whose K\"ahler cone is projectively induced, in contrast to the Sasaki--Einstein case studied in Theorem \ref{main}. This occurs since the Sasaki--Einstein condition is very rigid and it transforms to $\eta$--Einstein under $\mathcal D_a$-homothetic transformations.\\

The paper is organized as follows. Sections 2 and 3 are devoted to recalling the main definitions and results we need, respectively in the K\"ahler and the Sasakian contexts. In Section 4 we describe the relation between the K\"ahler potential on the cone over a Sasakian manifold and that of the transverse metric. Finally in Section 5 we prove the main theorems.

\medskip
\noindent
\emph{Acknowledgments:} The authors are grateful to Andrea Loi and Giovanni Placini for useful discussions on this paper. The authors would also like to thank the anonymous referees for a very careful and deep reading of the paper.

\section{Projectively induced K\"ahler manifolds}

We start by recalling some well-known facts about projectively induced K\"ahler manifolds and by fixing some notation.
Let $(M,g)$ be an $n$-dimensional K\"ahler manifold and denote by $\omega$ its associated $(1,1)$-form. Recall that, given any point $p\in M$, the condition of being K\"ahler is equivalent to the existence, in a neighborhood $U_p$ of $p$, of a real valued function $\varphi\!:U_p\rightarrow \mathds R$ called {\em K\"ahler potential} for $\omega$ such that $\omega_{|_{U_p}}=\frac i2\partial\bar\partial \varphi$. A K\"ahler potential is not unique, as any other function obtained from $\varphi$ by adding the real part of a holomorphic function is also a K\"ahler potential.
We say that $(M,g)$ is {\em projectively induced} if and only if  for all $p\in M$ there exists a holomorphic and isometric (i.e. {K\"ahler}) immersion $f\!:U_p\rightarrow \mathds C{\rm P}^N$, $N\leq \infty$, such that $f^*\omega_{FS}=\omega$.  Here $\mathds C{\rm P}^\infty$ denotes the infinite-dimensional complex projective space, namely the set of equivalence classes with the usual equivalence relation of points $(Z_0,Z_1,\dots)\in l^2(\mathds C)\setminus\{0\}$, endowed with the Fubini--Study metric $\omega_{FS}$ (see \cite{calabi}), where by $l^2(\mathds C)$ we denote the Hilbert space of sequences of complex numbers $z=(z_0,z_1,z_2,\dots)$ limited in norm, i.e. $\sum_{j=0}^\infty|z_j|^2<\infty$. In the sequel we will also denote by $\omega_0$ the flat metric on $\ell^2(\mathds C)$, that is $\omega_0=\frac i2\partial \bar \partial ||z||^2=\frac i2\sum_{j=0}^\infty dz_j\wedge d\bar z_j$. 
Recall that if $[Z_0:\dots:Z_j:\dots]$ are homogeneous coordinates and $(z_1,\dots, z_j,\dots)$, $z_j:=\frac{Z_j}{Z_0}$, are affine coordinates on $U_0:=\{Z_0\neq 0\}$, the Fubini--Study metric reads on $U_0$ as
$$
\omega_{FS{}_{|_{U_0}}}:=\frac i2\partial\bar\partial \log(1+||z||^2).
$$
We can restrict ourselves to considering real analytic K\"ahler metrics, since the pull-back by a holomorphic map of the Fubini--Study metric is necessarily real analytic and K\"ahler.
 Thus, one can define the {\em diastasis function} $D\!:U_p\times {U_p}\rightarrow \mathds R$ by:
$$
D(z,w):=\tilde\varphi(z,\bar z)+\tilde\varphi(w,\bar w)-\tilde\varphi(z,\bar w)-\tilde\varphi(w,\bar z),
$$
where  $\tilde\varphi$ is the function obtained by extending the  K\"ahler potential $\varphi$ on a
neighborhood of the diagonal in $U_p\times \overline{U_p}$ (here $\overline{U}_p$ denotes the neighborhood defined by conjugated local coordinates).
It is easy to see that fixing one of the two variables, the function $D_w(z):=D(z,w)$ is a K\"ahler potential for $g$, called diastasis function {\em centered at $w$}. 
The diastasis function has been introduced by Calabi in \cite{calabi} for its property of being - once one of the variables is fixed - a K\"ahler potential invariant by pull-backs by holomorphic maps, i.e. if $(\tilde M,\tilde\omega)$ is a second K\"ahler manifold and $f\!:U_p\rightarrow \tilde M$ is holomorphic and such that $f^*\tilde\omega=\omega$ we have:
$$
\frac{\partial^2 D_p(z)}{\partial z_j\partial \bar z_k}=\frac{\partial^2 (\tilde D_{f(p)}\circ f)(z)}{\partial z_j\partial \bar z_k}
$$
where we denote by $D_p(z)$ the diastasis function of $M$ on $U_p$ centered at $p$ and by $\tilde D_{f(p)}$ the diastasis function for $\tilde M$ in a neighborhood of $f(p)$. 

Let us denote by 
$(b_{jk})$ the $\infty \times \infty$ matrix of coefficients of the following power expansion:
\begin{equation}\label{expansion}
e^{D_p(z)}-1:=\sum_{j,k=0}^\infty b_{jk}z^{m_j}\bar z^{m_k},
\end{equation}
where we use a multi--index notation $m_j=(m_{j,1},\dots, m_{j,n})$, $z^{m_j}=z_1^{m_{j,1}}\cdots z_n^{m_{j,n}}$, and, setting $|m_j|:=m_{j,1}+\cdots +m_{j,n}$, the strings $m_j$'s are ordered in such a way that $m_0=(0,\dots,0)$, $|m_j|\leq |m_k|$ for $j<k$, and when $|m_j|=|m_k|$ the order can be taken e.g. to be lexicographic.
Furthermore,
\begin{equation}\label{bjkexplicit}
b_{jk}:=\frac{1}{m_j!m_k!}\frac{\partial^{|m_j|+|m_k|}}{\partial z^{m_j}\partial \bar z^{m_k}}\left(e^{D_p(z)}-1\right),
\end{equation}
where $m_j!=m_{j,1}!\cdots m_{j,n}!$.

The importance of the diastasis function relies on the following theorem, which summarizes the results for projectively induced metrics proved by Calabi in his celebrated work \cite{calabi}.
\begin{theor}[E. Calabi {\cite[Theorems 8 and 9]{calabi}}]\label{calabicentralmatrix}
Let $(M,g)$ be a real analytic K\"ahler manifold. 
 Then, $(M,g)$ admits a local K\"ahler immersion into $\mathds C{\rm P}^N$ if and only if the matrix $(b_{jk})$ defined by \eqref{expansion} is positive semidefinite of rank at most N.
Furthermore, the K\"ahler immersion is unique up to unitary transformations of the ambient space.
\end{theor}

\begin{remark}\label{globalcalabi}\rm 
It is worth pointing out that the condition of being positive semidefinite as well as the rank of $(b_{jk})$ do not depend on the point $p$ but only on the metric $g$. Furthermore, the rank of $(b_{jk})$ is exactly $N$ when the immersion is {\em full}, i.e. when the image is not contained in any lower dimensional complex projective space. Observe also that when $M$ is simply connected, the local K\"ahler immersions glue together to a global K\"ahler immersion of the whole $M$ (see \cite{calabi}). Finally a K\"ahler immersion of a compact K\"ahler manifold into $\mathds C{\rm P}^\infty$ is never full, as the immersion map is constructed using a basis of the space of global holomorphic sections of some holomorphic line bundle, and when $M$ is compact such space is finite dimensional.

\end{remark}

In this paper we are interested in the case when $(M,\omega)$ is compact and K\"ahler--Einstein (from now on KE). If we denote by $\rho$  the Ricci form associated to $\omega$, which can be written as $\rho=-i\partial \bar \partial \log\det g$, then the KE condition reads $\rho=\lambda\omega$, for $\lambda\in \mathds R$.

In \cite{H}  D. Hulin studied compact projectively induced K\"ahler--Einstein manifolds proving the following:
\begin{theor}[{\cite[Theorem 7.1]{H}}]\label{hulin}
Let $(M, g)$ be a compact projectively induced K\"ahler--Einstein manifold of complex dimension $n$. Then $M$ is simply connected, the immersion is global and the Einstein constant $\lambda$ is rational. Furthermore, if we write $\lambda = 2p/q > 0$, where $p/q$ is irreducible, then $p \leq n + 1$, and:
\begin{enumerate}
\item[(i)] if $p = n + 1$, then $(M,g) = (\mathds C{\rm P}^n,qg_{FS})$;
\item[(ii)] if $p = n$, then $(M,g) = (Q_n,qg_{FS})$, where $Q_n$ denotes the quadric in $\mathds C{\rm P}^{n+1}$ of homogeneous equation $Z_0^2+\cdots+Z_{n+1}^2=0$.
\end{enumerate}
\end{theor}

\section{Preliminaries on Sasakian geometry}

We devote this section to recall the main definitions and results about Sasakian manifolds that we will need in the following.
We mainly refer to \cite{BG,S} and references therein.

\begin{defin}
A Riemannian manifold $(S,g)$ is Sasakian if and only if its metric cone $(C(S), \bar g)$ is K\"ahler, where $C(S):=S\times \mathds R^+$ and $\bar g:= dt^2+t^2g$, for $t\in \mathds R^+$.
\end{defin}

It follows that $\dim_{\mathds R}S=2n+1$, where $\dim_{\mathds C}(C(S))=n+1$, and the K\"ahler structure on the cone induces on $S$:
\begin{enumerate}
\item a contact $1$-form $\eta$, i.e. $\eta\wedge (d\eta)^n\neq 0$, that is the restriction to $S$ of the form $d^c\log t$ on the cone.
\item a vector field $\xi$, uniquely defined by the conditions $\eta(\xi)=1$ and $d\eta(\xi,\,\cdot\,)=0$, named {\em Reeb vector field} of $\eta$;
\item  a $(1,1)-$tensor field $\Phi$ which satisfies $\Phi^2=-\mathbb{I}+\xi\otimes\eta$, that is the restriction to the contact distribution $\ker \eta$ of the complex structure of the cone.
\end{enumerate}
It follows that for any vector fields $X, Y$ on $S$ we have 
$$g(\Phi(X),\Phi(Y))=g(X,Y)-\eta(X)\eta(Y),$$
and
$$
d\eta(\Phi X,\Phi Y)=d\eta(X,Y),\quad
d\eta(\Phi X,X)>0,\;\forall X\neq 0.
$$
Let us denote by $\mathcal F$ the {\em Reeb foliation} defined by the integral curves of the Reeb vector field $\xi$. If we denote by $T_{\mathcal F}$ the tangent bundle to $\mathcal F$, the tangent bundle to $S$ splits canonically as $TS=\mathcal D\oplus T_{\mathcal F}$, where $\mathcal D:=\text{ker}\,\eta$ is the codimension one subbundle, with natural almost complex structure defined by $J:=\Phi |_{\mathcal D}$.
 Accordingly, the metric $g$ decomposes as $g=g^T+ \eta\otimes \eta$, where $g^T$ is a metric on the contact distribution called {\em transverse K\"ahler metric}, which can be globally identified with $g^T(\cdot,\cdot)=d\eta(\cdot,\Phi\cdot)$. 

A Sasakian manifold is said to be {\em regular}, {\em quasi-regular} or {\em irregular} if the foliation defined by the Reeb vector field has the corresponding property,  i.e.  respectively each leaves of the Reeb foliation intersect  every  neighborhood of a point  exactly one time, a finite number of times or neither of the previous two cases.
In this paper we are interested in compact regular Sasakian manifolds, for which we recall the following fundamental result (see e.g.  \cite[Theorem 7.5.1]{BG}):
\begin{theor}[Structure Theorem]\label{TheorStructure}
Let $(S,\xi,\eta,\Phi,g)$ be a compact regular Sasakian manifold. Then the space of leaves of the Reeb foliation $(X,\omega)$ is a compact K\"ahler manifold with integral K\"ahler form $\frac{1}{2\pi}\omega$ 
so that the projection $p:(S,g)\rightarrow(X,g_\omega)$ is a Riemannian submersion, where $g_\omega$ is the metric associated to $\omega$.

Viceversa, any principal $S^1$-bundle $S$ with Euler class $-\frac{1}{2\pi}[\omega]\in H^2(X,\mathds Z)$ over a compact K\"ahler manifold $(X,\omega)$ admits a Sasakian structure.
\end{theor}

We will call $(X,g_\omega)$ in Theorem \ref{TheorStructure}  the \emph{K\"ahler manifold associated} to $S$.

This theorem allows us to describe the K\"ahler structure of the cone over $S$ through the dual of a positive line bundle over $X$ in the following way (see e.g. \cite[p. 6]{LPZ}). Assume that $S$ is  a principal circle bundle $p\!:S\rightarrow X$ over a compact Hodge K\"ahler manifold $(X,\omega)$, namely a compact manifold $X$ with integral K\"ahler form $\omega$, such that $p^*\omega=\frac 12d\eta$. The K\"ahler class of $\omega$ defines an ample line bundle $L$ over $X$ such that $C(S)=L^{-1}\setminus\{0\}$, where $L^{-1}$ is the line bundle dual to $L$. Furthermore, there exists a Hermitian metric $h$ on $L$ such that $\omega=-i\partial \bar\partial \log h$ and the dual metric $h^{-1}$ on $L^{-1}$ defines the radial coordinate of the cone in the following way:
\begin{equation}\label{t}
t\!:L^{-1}\setminus \{0\}\rightarrow \mathds R^+,\quad (x,v)\mapsto |v|_{h_x^{-1}} \ \ (v\in L_x^{-1}).
\end{equation}
The K\"ahler metric $\bar g$ on the cone has K\"ahler form $ \Omega=\frac i2\partial\bar \partial t^2$.

We recall that for $a>0$, by a $\mathcal{D}_a$\textit{-homothetic} trasformation 
of $(S,\xi,\eta,\Phi,g)$ we
mean a change of the structure tensors in the following way (see \cite{tanno}):
\begin{equation}\label{da}
\eta_a=a\eta,\quad \xi_a=\frac{1}{a}\xi,\quad \Phi_a=\Phi, \quad g_a=ag+a(a-1)\eta\otimes\eta.
\end{equation}
It is important to notice that if we perform a $\mathcal D_a$-homothetic transformation of the Sasakian structure $(\xi,\eta,\Phi,g)$ over $S$, then the new Sasakian structure is still regular and the complex manifold below is still $X$, but its K\"ahler form $\omega_a$ is $\omega$ rescaled by $a$, so that $p^*\omega_a=ap^*\omega=\frac a2d\eta$.

Observe that setting a new coordinate ${t'}=t^a$ on the K\"ahler cone induces the same structure $(\xi_a,\eta_a,\Phi_a,g_a)$ on $S$.\par

We conclude this section with the following theorem (see for instance \cite[Th. 11.1.3, Lemma 11.1.5, Cor. 11.1.8]{BG}) that summarizes what we need about Sasakian--Einstein manifolds, namely Sasakian manifolds whose metric $g$ is Einstein in the Riemannian sense, i.e. ${\rm Ric}_g=\lambda g$, for a constant $\lambda$.
\begin{theor}\label{SE}
Let $(S,\xi,\eta,\Phi,g)$ be a Sasakian manifold, then the following are equivalent:
\begin{enumerate}
\item the metric $g$ is Sasakian--Einstein with Einstein constant $2n$;
\item the metric $\bar g$ on the K\"ahler cone is Ricci--flat.
\end{enumerate}
In addition, if $S$ is compact with associated K\"ahler manifold $(X, \omega)$, then the previous conditions are equivalent to $\omega$ being K\"ahler--Einstein with Einstein constant $2n+2$.
\end{theor}

\section{K\"ahler potentials related to regular Sasakian manifolds}

Consider a compact regular Sasakian manifold $(S,\xi,\eta,\Phi,g)$. By  Theorem \ref{TheorStructure}, $S$ is a principal circle bundle $p\!:S\rightarrow X$ over a compact Hodge K\"ahler manifold $(X,\omega)$ such that $p^*\omega=\frac{a}{2}d\eta$, for some $a>0$.%, {\color{blue}where the notation $\omega_a$ for the integral form on $X$ is justified by the notation in the next lemma.}

\begin{lemma}\label{giulio}
Let $(S,\xi,\eta,\Phi,g)$ be a compact regular Sasakian manifold and let $(X,\omega_a)$ be its associated K\"ahler manifold.
Assume $p^*\omega=\frac a2d\eta$, $a>0$. If $\omega=\frac i2\partial\bar \partial \psi$ on a open set $U$ around $x\in X$, then the K\"ahler metric $\Omega$ on the cone over $S$ is given by $\Omega=\frac i2\partial\bar \partial \!\left(|z_0|^{\frac2a}e^{\frac1a \psi}\right)$ on $ \mathds C\setminus\{0\}\times U\subset C(S)$, where $z_0$ is the coordinate on $\mathds C\setminus \{0\}$.
\end{lemma}
\begin{proof}
Following the line bundle construction in the previous section, $S$ is a principal circle bundle over a compact Hodge K\"ahler manifold $(X,\omega)$ and we can construct an ample line bundle $\pi\!: L\rightarrow X$ over $X$ endowed with a Hermitian metric $h$ such that $\omega=-i\partial\bar \partial\log  h$. Observe that $\Omega_a=\frac i2\partial \bar \partial t^{2}$ is the K\"ahler metric on the cone over the Sasakian manifold $(S,\xi_a, \eta_a,\Phi,g_a)$, where $t$ is given by \eqref{t}. Since to move from $\eta_a$ to $\eta$ we need a $\mathcal D_{1/a}$ homothetic transformation, which corresponds on the cone to the change of variable $t'=t^{1/a}$, the K\"ahler metric on the cone over $(S,\xi,\eta,\Phi,g)$ is $\Omega=\frac i2\partial \bar \partial t^{2/a}$.

Consider a trivialization $\{U,\sigma\}$ of the line bundle $L$, where $U$ is an open neighborhood of $x\in X$ and $\sigma\!:U\rightarrow \pi^{-1}(U)$, $x\mapsto (x,\sigma(x))$, is a trivializing section. Recall that $\sigma$ defines an isomorphism $\pi^{-1}(U)\rightarrow  \mathds C\times U$ by $(w,x)\mapsto (\alpha_w,x)$, with $w\in L_x$, and $\alpha_w\in \mathds C$ defined by $w=\alpha_w\sigma(x)$. Thus, any other section $s\!:U\rightarrow \pi^{-1}(U)$, $x\mapsto (s(x),x)$ can be written as $s(x)=f(x)\sigma(x)$ with $f\!:U\rightarrow \mathds C$ holomorphic. 

On the dual bundle $\hat \pi\!:L^{-1}\rightarrow X$ denote by $\sigma^{-1}$ the section dual to $\sigma$, i.e. $\sigma^{-1}\!:U\rightarrow \hat \pi^{-1}(U)$ with $\sigma^{-1}(x)\!:L_x\rightarrow \mathds C$ satisfying $\sigma^{-1}(x)(\sigma(x))=1$. By linearity, for any $w=\alpha_w\sigma(x)\in L_x$,  $\sigma^{-1}(x)(w)=\sigma^{-1}(x)(\alpha_w\sigma(x))=\alpha_w$. As before any other section can be written as $\hat s(x)=\hat f(x)\sigma^{-1}(x):L_x\rightarrow \mathds C$. Observe that evaluating at $\sigma(x)$ we get $\hat f(x)=\hat s(x)(\sigma(x))$. In other words, in the coordinate system defined by $\sigma^{-1}$, the coordinate of $\hat s(x)$ is $\hat s(x)(\sigma(x))$.

The dual Hermitian metric $h^{-1}$ is defined on $x\in U$ by:
$$
h_x^{-1}\!:L_x^{-1}\times L_x^{-1}\rightarrow \mathds C,\quad h_x^{-1}(\hat s(x),\hat s(x)):=h_x(h_x^{\sharp}(\hat s(x)),h_x^{\sharp}(\hat s(x))),
$$
where $h^\sharp\!:L_x^{-1}\rightarrow L_x$ is the inverse of the canonical isomorphism induced by $h$ on the fibers, $h_x^\flat\!:L_x\rightarrow  L_x^{-1}$, $w\mapsto h_x(w,\cdot)$. Since $h_x^{\sharp}(\hat s(x))\in L_x$ we have for some $\beta\in \mathds C$,  $h_x^{\sharp}(\hat s(x))=\beta \sigma(x)$. Furthermore, the function $h_x(w,\cdot)\!: L_x\rightarrow \mathds C$ applied to $\sigma(x)$ gives:
$$
h_x(w,\sigma(x))=h_x(\alpha_w\sigma(x),\sigma(x))=\alpha_w  h_x(\sigma(x),\sigma(x)),
$$
so applying $h_x^\flat(h_x^{\sharp}(\hat s(x)))=\hat s(x)$ to $\sigma(x)$ we get:
$$
h_x^\flat(h_x^{\sharp}(\hat s(x)))(\sigma(x))=\beta h_x(\sigma(x),\sigma(x))=\hat s(x)(\sigma(x)).
$$
Thus,
$$
h_x^{\sharp}(\hat s(x))=\frac{\hat s(x)(\sigma(x))}{h_x(\sigma(x),\sigma(x))}\sigma(x),
$$
which implies:
$$
h_x^{-1}(\hat s(x),\hat s(x))=\frac{|\hat s(x)(\sigma(x))|^2}{h_x(\sigma(x),\sigma(x))}.
$$
Then in \eqref{t}, if we denote by $z_0$ the coordinate of a vector $v\in L_x^{-1}$ with respect to $\sigma^{-1}$ we have:\\
$$
|v|^{2}_{h_x^{-1}}=h_x^{-1}(v,v)=\frac{|v(\sigma(x))|^2}{h_x(\sigma(x),\sigma(x))}=\frac{|z_0|^2}{h_x(\sigma(x),\sigma(x))},
$$
and on $U$ and $\mathds C\setminus\{0\}\times U$, we have:
$$
\Omega=\frac i2\partial\bar\partial \frac{|z_0|^{2/a}}{h_x(\sigma(x),\sigma(x))^{1/a}},\quad \omega=-\frac i2\partial\bar \partial \log h_x(\sigma(x),\sigma(x)).
$$
Notice that this construction is independent of the choice of the trivializing section, as the linear bundle is Hermitian and thus the transition functions are unitary.
The conclusion follows by noticing that $\psi= \log h_x(\sigma(x),\sigma(x))^{-1}$ implies $|z_0|^{\frac2a}e^{\frac1a\psi}=\frac{|z_0|^{\frac2a}}{h_x(\sigma(x),\sigma(x))^{\frac1a}}$.
\end{proof}

\begin{ex}\label{hey}\rm
If $\omega=\omega_{FS}$  is the Fubini-Study metric on $X=\mathbb{C}P^n$, then $\psi=\log(1+||z||^2)$ lifts to $|z_0|^2(1+||z||^2)$. The K\"ahler cone corresponding to $(\mathds C{\rm P}^n,\omega_{FS})$ is $(\mathds C^{n+1}\setminus\{0\},\omega_0)$. In our notation one obtains the canonical potential for the flat metric, i.e. $|z_0|^2+||z||^2$, after performing the holomorphic change of variables $(z_0,z_1,\dots, z_n)\mapsto \left(z_0,\frac{z_1}{z_0},\dots, \frac{z_n}{z_0}\right)$.\end{ex}

\begin{remark}\rm
It is worth pointing out that a similar construction holds in the noncompact setting. In fact, due to \cite[Sec. 5]{LPZ} a regular noncompact Sasakian manifold is a principal bundle over a {\em noncompact} K\"ahler manifold $(X,\omega)$. More precisely, the fiber of a noncompact Sasakian manifold over a compact $X$ would be $\mathds R$, the principal bundle is forced to be trivial and the K\"ahler form on $X$ must be exact, which is impossibile if $X$ is compact. Thus a Sasaki noncompact manifold is either a circle bundle over a noncompact K\"ahler manifold $(X,\omega)$ or it is the product $X\times \mathds R$.
In both cases, if $\omega$ is integral, we construct a positive Hermitian line bundle $(L,h)$ over $X$ such that $\omega=-i\partial\bar\partial h$. Then Lemma \ref{giulio} holds for noncompact Sasakian manifolds.
\end{remark}
\section{Projectively induced K\"ahler cones}

We begin by proving the following proposition, interesting in its own sake and a key step in the proof of our main theorems.
\begin{prop}\label{amelia}
Let $U$ be  a contractible domain in $\mathds C^n$, endowed with a K\"ahler metric $\omega=\frac i2\partial \bar \partial \psi$. Consider on $ \mathds C\setminus\{0\}\times U$ the K\"ahler metric $ \Omega_c=\frac i2\partial\bar \partial (|z_0|^{2c}e^{c\psi})$ for a constant $c>0$, where $z_0$ is the coordinate on $\mathds C\setminus \{0\}$. Then for any $c>0$, $\Omega_c$ is projectively induced if and only if $c\,\omega$ is.
\end{prop}
\begin{proof} Observe that $\Omega_c$ is real analytic if and only if $c\omega$ is.
Set coordinates $z$ around a point $p\in U$ such that $z(p)=0$.

Assume first that $c\omega$ is projectively induced, i.e. there exist holomorphic functions $f_j\!:U\rightarrow \mathds C$ such that $e^{c\psi}=\sum_{j=0}^\infty |f_j|^2$, with $f_0\equiv1$ and for all $j\geq 1$, $f_j(0)=0$. Then we have,
$$
|z_0|^{2c}e^{c\psi}=|z_0|^{2c}\sum_{j=0}^\infty |f_j|^2=\sum_{j=0}^\infty |z_0^cf_j|^2,
$$ 
and thus:
$$
e^{|z_0|^{2c}e^{c\psi}}=\sum_{j=0}^\infty|k_j(z_0,z)|^2
$$
for some holomorphic functions $k_j$ defined on $\mathds C\setminus\{0\}\times U$ and such that $k_0=e^{|z_0|^{2c}}$. It follows that $k\!:\mathds C\setminus\{0\}\times U\rightarrow \mathds C{\rm P}^\infty$, $k(z_0,z)=\left[k_0:k_1:\dots:k_j:\dots\right]$ is a K\"ahler immersion and thus $(\mathds C\setminus\{0\}\times U,\Omega_c)$ is projectively induced.

Assume now $\Omega_c$ to be projectively induced.  We can assume without loss of generality that $\psi$ is the diastasis function for $\omega$ centered at $p$, i.e. $\psi(0)=0$. By Theorem \ref{calabicentralmatrix}, the matrix of coefficients $(B_{jk})$ in the power expansion around the origin of the diastasis function of $\Omega_c$
is positive semidefinite. We need to show that $(B_{jk})\geq 0$ implies $(b_{jk})\geq 0$, where $(b_{jk})$ is the matrix of coefficients in the power expansion of $e^{c\psi}-1$ (cf. \eqref{expansion}). 

Let  $q=(\epsilon,p)\in \mathds C\setminus\{0\}\times U$, with $\epsilon\in \mathds R^+$. Perform a change of coordinates $z_0\rightarrow z_0-\epsilon$ (we avoid to change the name of the variable to simplify the notation), so that $q$ is the point of coordinates $(0,0)$. Observe that in this coordinates a potential for $\Omega_c$ reads:
$$
\Phi(z_0,z)= |z_0+\epsilon|^{2c}e^{c\psi(z)},
$$
and the diastasis for $\Omega_c$ centered at $q$ is given by:
$$
D_q(z_0,z)= |z_0+\epsilon|^{2c}e^{c\psi\left(z\right)}- \epsilon^c(z_0+\epsilon)^c-\epsilon^c(\bar z_0+\epsilon)^c+\epsilon^2.
$$
Observe that if $(B_{jk})$ is positive semidefinite, then any submatrix obtained from $(B_{jk})$ collecting the $a_1,\dots, a_l$-th rows and columns also is. In particular the matrix $(u_{jk})$ has to be positive semidefinite, where
\begin{equation}\label{francesca}
u_{jk}=\left[\frac{1}{\hat m_j!\hat m_k!}\frac{\partial^{|\hat m_j|+|\hat m_k|}}{\partial z^{\hat m_j}\bar z^{\hat m_k}}\left(\frac{\partial^{2}}{\partial z_0\partial \bar z_0}\left(e^{D_q(z_0,z)}-1\right)\right)_{z_0=0}\right]_{z=0},
\end{equation}
where comparing with the multi-index notation in \eqref{bjkexplicit} we are setting $m_j=(1,\hat m_j)$, since we are isolating the derivatives with respect to $z_0$ and $\bar z_0$.
By direct computation:
\begin{equation}
\begin{split}
&\left[\frac{\partial^{2}}{\partial z_0\partial \bar z_0}\left(e^{D_q(z_0,z)}-1\right)\right]_{{z_0=0}}
=\left[\frac{\partial}{\partial z_0}\left\{e^{D_q(z_0,z)}\left(c(z_0+\epsilon)^c(\bar z_0+\epsilon)^{c-1}e^{c\psi\left(z\right)}- c\epsilon^c(\bar z_0+\epsilon)^{c-1}\right)\right\}\right]_{{z_0=0}}\\
&=\left[e^{D_q(z_0,z)}\left\{c^2(z_0+\epsilon)^{c-1}(\bar z_0+\epsilon)^{c-1}e^{c\psi\left(z\right)}+\right.\right.\\
&\left.\left.+\left((z_0+\epsilon)^cc(\bar z_0+\epsilon)^{c-1}e^{c\psi\left(z\right)}-\epsilon^cc(\bar z_0+\epsilon)^{c-1}\right)\left((\bar z_0+\epsilon)^cc( z_0+\epsilon)^{c-1}e^{c\psi\left(z\right)}-\epsilon^cc( z_0+\epsilon)^{c-1}\right)\right\}\right]_{z_0=0}\\
   &=c^2\epsilon^{2c-2}e^{D_q(0,z)}\left\{ 
\epsilon^{2c}\left(e^{c\psi\left(z\right)}- 1\right)^2
+e^{c\psi(z)}
   \right\}.
\end{split}\nonumber
\end{equation}
The matrix $(u_{jk})$ is positive semidefinite if and only if $(v_{jk}):=c^{-2}\epsilon^{2-2c}(u_{jk})$ is, for any $\epsilon>0$. 
Since
$$
v_{jk}=\left[\frac{1}{\hat m_j!\hat m_k!}\frac{\partial^{|\hat m_j|+|\hat m_k|}}{\partial z^{\hat m_j}\bar z^{\hat m_k}}e^{D_q(0,z)}\left\{ \epsilon^{2c}\left(e^{c\psi\left(z\right)}- 1\right)^2
+e^{c\psi(z)}
   \right\}\right]_{z=0},
$$
is positive semidefinite for any $\epsilon>0$, it must remain positive semidefinite as $\epsilon$ goes to 0. The conclusion follows observing that for $\epsilon=0$ one has
$$(v_{jk})= \left[\frac{1}{\hat m_j!\hat m_k!}\frac{\partial^{|\hat m_j|+|\hat m_k|}}{\partial z^{\hat m_j}\bar z^{\hat m_k}}e^{c\psi(z)}\right]_{z=0}=\left[\frac{1}{\hat m_j!\hat m_k!}\frac{\partial^{|\hat m_j|+|\hat m_k|}}{\partial z^{\hat m_j}\bar z^{\hat m_k}}\left(e^{c\psi(z)}-1\right)\right]_{z=0}=(b_{jk}).$$
\end{proof}
%{\color{blue}\begin{remark}\rm
%Observe that since performing a $\mathcal D_a$--homothetic transformation of $(S,\eta)$ changes the potential $t^2$ on the K\"ahler cone as $\hat t^2=t^{2a}$, when $t^2=|z_0|^2e^{\psi}$ we get $\hat t^2=|z_0|^{2a}e^{a\psi}$. When $a\notin \mathds Z^+$ this potential has no hope to be associated to a projectively induced metric, since the coefficients in the power expansion at $z_0=\epsilon$:
%$$
%|z_0|^{2a}=\sum_{j=0}^\infty\frac{\epsilon^{a-j}|z_0|^{2j}}{j!}\prod_{k=0}^{j-1}(a-k),
%$$
%are nonnegative if and only if $a\in \mathds Z^+$.
%\end{remark}}
%

\begin{proof}[Proof of Theorem \ref{main}]
If the metric is flat it is trivially Ricci-flat, furthermore by a result of Calabi \cite{calabi} it is also projectively induced. Assume now  that $C(S)$ is Ricci--flat and projectively induced. By Theorem \ref{SE}
the metric cone $(C(S),\bar g, \Omega)$ is Ricci--flat if and only if $(S,\xi,\eta,\Phi,g)$ is Sasaki Einstein and
by Myers' Theorem, a complete Sasaki--Einstein manifold is forced to be compact. By the Structure Theorem \ref{TheorStructure} a compact regular Sasakian manifold is a principal circle bundle over a compact Hodge manifold $(X,\omega)$. Denote by $p\!:S\rightarrow X$ the bundle projection. 
Recall that the transverse K\"ahler metric $g^T$ is K\"ahler--Einstein with Einstein constant $\lambda^T=2(n+1)$ (cf. Theorem \ref{SE}). Since for some $a>0$, $p^*\omega=\frac a2d\eta$, and $\frac12d\eta$ is the K\"ahler form associated to $g^T$, then $(X,\frac 1a\omega)$ is a compact K\"ahler--Einstein manifold with Einstein constant $2(n+1)$.

By Proposition \ref{amelia}, for any $c>0$, $c\omega$ is projectively induced if and only if $\Omega_c$ is. Since by Lemma \ref{giulio}, $\Omega=\Omega_{1/a}$, we obtain that $\frac1a\omega$ is projectively induced if and only if $\Omega$ is. Observe that the K\"ahler immersion of $(X,\frac 1a\omega)$ in $\mathds C{\rm P}^N$ is also global since the manifold is Fano and thus simply connected (see Remark \ref{globalcalabi}).

  By Hulin's Theorem \ref{hulin}, a projectively induced K\"ahler-Einstein compact manifold with Einstein constant $2(n+1)$ is forced to be the complex projective space with the Fubini--Study metric, i.e. $(X,\frac 1a\omega)\simeq(\mathds C{\rm P}^n,\omega_{FS})$, 
and the Sasakian manifold is Sasaki equivalent to the sphere $\mathds S^{2n+1}$ with the standard Sasakian structure, whose cone is the complex Euclidean space with its flat metric $(C(S),\Omega)\simeq (\mathds C^{n+1}\setminus \{0\},\omega_0)$  (see \cite[Ex. 5]{CML} and Example \ref{hey} above). 
\end{proof}
\begin{remark}\rm
Observe that instead of applying Hulin's Theorem, in the last part of the proof of Theorem \ref{main} we could have argued in the following way. The global K\"ahler immersion of $(X,\frac1a \omega)$ into $\mathds C{\rm P}^{N<\infty}$ induces by \cite[Proposition 3]{CML} a Sasakian immersion of $S$ in the Sasakian sphere $\mathds S^{2N+1}$. Thus, by \cite[Theorem 1]{CML}, $S$ is forced to be Sasakian equivalent to a standard Sasakian sphere and its cone to be biholomorphically isometric to $\mathds C^{n+1}\setminus\{0\}$ with the flat metric.  
\end{remark}
We conclude the paper with the proof of Theorem \ref{main2} showing that, up to $\mathcal D_a$--homothetic transformation, K\"ahler cones over homogeneous compact Sasakian manifolds are projectively induced. 

\begin{proof}[Proof of Theorem \ref{main2}]
By \cite[Theorem 8.3.6]{BG}, a compact homogeneous Sasakian manifold is regular and fibers over a simply connected homogeneous K\"ahler manifold $(X,\omega)$. By Theorem \ref{TheorStructure}, $\omega$ is Hodge and $p^*\omega=\frac a2d\eta$, where $p\!:S\rightarrow X$ is the fiber projection. By Lemma \ref{giulio}, if locally $\omega=\frac i2\partial\bar \partial \psi$, then locally the K\"ahler form on the cone reads $\Omega=\frac i2\partial\bar \partial |z_0|^{\frac2a} e^{\frac 1a\psi}$.

%Observe that $\Omega_c= \frac i2\partial\bar \partial |z_0|^2 e^{c\psi}$ is the metric on $C(S)$ obtained performing a $\mathcal D_{ac}$--homothetic transformation on $S$, since if $t^2=|z_0|^2 e^{\frac 1a\psi}$, then $t^{2ac}=|z_0|^{2ac} e^{c\psi}\simeq|z_0|^{2} e^{c\psi}$, where the last equivalence holds in view of a change of variable on $\mathds C\setminus\{0\}$ such that $|z_0|^{2ac}\rightarrow |z_0|^2$.using Dynkin diagrams of semisimple Lie groups,

Recall that in \cite[Section 2]{take} Takeuchi proved 
that if the K\"ahler form $\omega$ of a compact and simply connected homogeneous K\"ahler manifold is integral, then there exists a positive integer $k$ such that $k\omega$ is projectively induced. The flag manifold $(X,\omega)$ satisfies  the hypothesis of Takeuchi's theorem, thus $k\omega$ admits a global K\"ahler immersion in $\mathds C{\rm P}^N$. 

Performing a $\mathcal D_{ak}$--homothetic transformation on $(S,\eta)$ the potential $t^2=|z_0|^{\frac2a} e^{\frac 1a\psi}$ on the K\"ahler cone changes accordingly in $t^{2ak}=|z_0|^{2k} e^{k\psi}$. %Let $z_0=\rho_0 e^{i\theta_0}$ in polar coordinates, and consider the {\color{red}holomorphic} change of variables that modifies $z_0$ alone in $\hat z_0=\rho_0^{ac}e^{i\theta_0}$. In the new coordinates, we have $t^{2ac}=|\hat z_0|^{2} e^{c\psi}$. 
By Proposition \ref{amelia}, $\frac i2\partial\bar \partial |z_0|^{2k} e^{k\psi}$ is projectively induced if and only if $k\omega$ is, and we are done.
\end{proof}

\begin{remark}\rm
Observe that in both the proofs of Theorem \ref{main} and Theorem \ref{main2}, the local arguments given by Lemma \ref{giulio} and Proposition \ref{amelia} could be extended to quasi regular Sasakian manifolds. However, in the quasi-regular case $X$ is a Hodge K\"ahler orbifold and in this setting we do not have  the techniques developed for K\"ahler immersions that are crucial in the concluding steps of our proofs.
\end{remark}

\end{document}